\documentclass[pdflatex,sn-mathphys-num]{sn-jnl}

\usepackage{graphicx}%
\usepackage{multirow}%
\usepackage{amsmath,amssymb,amsfonts}%
\usepackage{amsthm}%
\usepackage{mathrsfs}
\usepackage{mathrsfs}
\usepackage[title]{appendix}%
\usepackage{xcolor}%
\usepackage{textcomp}%
\usepackage{manyfoot}%
\usepackage{booktabs}%
\usepackage{algorithm}%
\usepackage{algorithmicx}%
\usepackage{algpseudocode}%
\usepackage{listings}%
\usepackage{framed}
\usepackage{nameref}
\usepackage{orcidlink}
\usepackage{amsmath, amsthm, amssymb, amsfonts, mathrsfs}
\usepackage{enumerate}
\usepackage{geometry}
\usepackage{mathtools}
\usepackage{hyperref}
\usepackage{cleveref}

\theoremstyle{plain}
\newtheorem{theorem}{Theorem}[section]
\newtheorem{proposition}[theorem]{Proposition}
\newtheorem{lemma}[theorem]{Lemma}

\theoremstyle{definition}
\newtheorem{definition}[theorem]{Definition}
\newtheorem{rem}[theorem]{Remark}
\newtheorem{example}[theorem]{Example}

\DeclareMathOperator{\ext}{\textup{ext}}

\DeclareMathOperator{\inter}{\textup{int}}
\DeclareMathOperator{\Per}{\textup{Per}}

\DeclareMathOperator{\dist}{\textup{dist}}
\DeclareMathOperator{\diam}{\textup{diam}}
\DeclareMathOperator{\difm}{\textup{d$\mathcal{H}$}}
\DeclareMathOperator{\dif}{\textup{d$\lambda$}}

\newcommand\myeqqq{\mathrel{\stackrel{\makebox[0pt]{\mbox{\normalfont\tiny (\ref{dr})}}}{\geq}}}
\newcommand\weq{\mathrel{\stackrel{\makebox[0pt]{\mbox{\normalfont\tiny (\ref{bez})}}}{\leq}}}
\newcommand\myequuh{\mathrel{\stackrel{\makebox[0pt]{\mbox{\normalfont\tiny (\ref{incl})}}}{\leq}}}\newcommand\myequuuh{\mathrel{\stackrel{\makebox[0pt]{\mbox{\normalfont\tiny (\ref{kik})}}}{\leq}}}
\newcommand\myequh{\mathrel{\stackrel{\makebox[0pt]{\mbox{\normalfont\tiny (\ref{haus})}}}{\leq}}}
\begin{document}

\title[Anisotropic Lower-Dimensional Minkowski Content and $\mathcal{S}$-content]{Anisotropic Lower-Dimensional Minkowski Content and $\mathcal{S}$-content}

\author*[]{\fnm{Filip} \sur{Fryš}\,\orcidlink{0009-0000-2134-0109}}\email{filip.frys@matfyz.cuni.cz}

\affil[]{\orgdiv{Mathematical Institute of Charles University}, \orgname{Charles University, Faculty of Mathematics and Physics}, \orgaddress{\street{Sokolovsk\'{a} 83}, \city{Praha 8}, \postcode{186 75}, \state{Czech Republic}}}

\abstract{This paper investigates the lower-dimensional anisotropic Minkowski content and 
$\mathcal{S}$-content. We establish that these anisotropic contents exhibit properties analogous to their isotropic counterparts by proving analogous inequalities between the lower-dimensional anisotropic Minkowski content and 
$\mathcal{S}$-content. A key component of our approach is demonstrating that the associated anisotropic volume function is of Kneser type, a result that underpins many of our proofs. In addition, we introduce anisotropic versions of the Minkowski and 
$\mathcal{S}$-dimensions and derive inequalities relating them. As an application, we analyze the existence of the 
$log_2(3)$-dimensional anisotropic Minkowski and 
$\mathcal{S}$-contents of the Sierpinski gasket.}




\maketitle\section*{Acknowledgment}{I would like to express my gratitude to my master's thesis supervisor Jan Rataj for his support and valuable advice, thanks to which I was able to complete my thesis, from which most of the results presented in this article are derived. 

The author gratefully acknowledges the financial support from \emph{SVV-2025-260837} and \emph{PRIMUS/24/SCI/009}, provided by Charles University}\section*{Introduction}\label{secIntro}
This paper is dedicated to anisotropic versions of lower-dimensional \textit{Minkowski content} and \textit{$\mathcal{S}$-content}. It provides a generalization of the notions of isotropic lower-dimensional Minkowski content and $\mathcal{S}$-content introduced in \cite{Rataj}. 

The so-called \textit{$C$-anisotropic volume function} is defined at the point $r\geq0$ as the Lebesgue measure of the set $E\oplus rC\coloneq\{e+rc;\;(e, c)\in E\times C\}$, that is $$V_{E, C}(r):=\lambda^n(E\oplus rC)\quad\text{for $r\geq0$},$$ where  $E\subseteq\mathbb{R}^n$ is a fixed compact set and $C\subseteq\mathbb{R}^n$ is a fixed \textit{convex body} the interior of which contains $0$, i.e. a compact convex set containing properly the origin. We can understand the function $V_{E, C}(\cdot)$ as a measurement of the volume of the \textit{$C$-anisotropic tubular neighborhood of $E$}, i.e. the set $E_{r, C}\coloneq E\oplus rC$, and so $$V_{E, C}(r)=\lambda^n(E_{r, C})\quad\text{for $r\geq0$.}$$

This function is studied in depth in \cite{Villa}. The authors characterize the set of points of differentiability of $V_{E, C}(\cdot)$, show that its complement is at most countable, give an explicit formula for its one-sided derivatives, and describe the second derivative of $V_{E, C}(\cdot)$ in the distributional sense. We will use their results to prove that for a given compact $E$, the $C$-anisotropic volume function \textit{of Kneser type of order n}, i.e. \begin{align}\label{knes}
    V_{E, C}(tb)-V_{E, C}(ta)\leq t^n\big(V_{E, C}(b)-V_{E, C}(a)\big)\quad \text{ for any $t\geq 1$ and $b>a>0$}.
\end{align}

Kneser proved in \cite{Kneser} that $V_{E, C}(\cdot)$ satisfies (\ref{knes}) in the isotropic case, i.e. for $C=B(0, 1)$. Later, 
Stachó showed in \cite{Stachó} that $V_{E, C}(\cdot)$ is Kneser for any symmetric $C$ and arbitrary $E$. But the results of \cite{Villa} enable us to provide a proof for a general convex body $C$ containing properly the origin.

\medskip

The \textit{$s$-dimensional (isotropic) Minkowski of content of a compact set $E\subseteq\mathbb{R}^n$} is given by \begin{equation}\label{eq.2}
    \mathcal{M}^s(E)\coloneq\lim_{r\to0_+}\frac{\lambda^n\big(E\oplus B(0, r)\big)}{\omega_{n-s}r^{n-s}}=\lim_{r\to0_+}\frac{V_{E, B(0, 1)}(r)}{\omega_{n-s}r^{n-s}}
\end{equation} if it exists, where $s\in[0, n]$ and $$\omega_{t}\coloneq\frac{\pi^{\frac{t}{2}}}{\Gamma(1+t)}\quad\text{for $t\geq0$.}$$ Here, $\Gamma(\cdot)$ stands for the standard \textit{Gamma function.} If $t\in\mathbb{N}_0$, then $\omega_t$ is nothing but the volume of the $t$-dimensional unit ball.

Similarly, if $s\in[0, n)$, then the \textit{$s$-dimensional (isotropic) $\mathcal{S}$-content of a compact set $E\subseteq\mathbb{R}^n$} is given by\begin{equation}\label{eq.1}
    \mathcal{S}^s(E)\coloneq\lim_{r\to0_+}\frac{V_{E, B(0,1)}^{\prime}(r)}{(n-s)\omega_{n-s}r^{n-s}},
\end{equation} if the limit exists.

The properties and behavior of (\ref{eq.2}) and (\ref{eq.1}) are studied in \cite{Rataj}, where the majority of the proofs rely on the fact that the function $V_{E, B(0, 1)}$ is of Kneser type. Since we prove that $V_{E, C}(\cdot)$ is of Kneser type for a general convex body $C$ of full dimension, we can naturally extend the results of \cite{Rataj}: we will study the convergence and asymptotic behavior of
 \begin{align}\label{kds2}
    \frac{V_{E, C}(r)}{r^{n-s}}&&\text{and}&&\frac{V_{E, C}^{\prime}(r)}{(n-s)r^{n-s-1}}
\end{align} as $r$ tends to zero from the right, where $s\in[0, n)$.

The \textit{$s$-dimensional $C$-anisotropic Minkowski content of $E$} is defined as the limit of the first item of (\ref{kds2}) as $r$ approaches zero from the right, if it exists, i.e. \begin{equation}\label{kds8}
\mathcal{M}^s_C(E)\coloneq\lim_{r\to0_+}\frac{V_{E, C}(r)}{r^{n-s}}.
\end{equation} 

The \textit{$s$-dimensional $C$-anisotropic $\mathcal{S}$-content of $E$} is given by \begin{equation}\label{kds3}
    \mathcal{S}^s_C(E)\coloneq\lim_{r\to0_+}\frac{V_{E, C}^{\prime}(r)}{(n-s)r^{n-s-1}},
\end{equation} if the limit exists, where $s\in[0, n)$.

We prove inequalities between limits superior and inferior as $r$ tends to zero from the right of the fractions in (\ref{kds8}) and compute them explicitly for the \textit{Sierpinki gasket} for arbitrary $C$. We will find that all four values are in this case different.

Later, we introduce the \textit{$C$-anisotropic Minkowski dimension and $\mathcal{S}$-dimension} and prove inequalities between them thereafter. The isotropic case (i.e. for $C=B(0, 1)$) is treated in the article \cite{Rataj}, which we draw inspiration from. We will also discuss the role of $C$ in the definitions of the dimensions mentioned above.

\medskip

At the end of this paper, we provide a proof of a theorem that generalizes Theorem 3.4 from \cite{Villa2}. It is a standard-type of an existence-theorem in this branch of mathematics---it basically says that the existence of a countable system of a given compact $E$ formed by pairwise disjoint compact sets that covers the set $E$ up to an $\mathcal{H}^s$-negligible set and each element of which admits the $s$-dimensional $C$-anisotropic Minkowski content ensures the existence of the $C$-anisotropic Minkowski content of the whole set $E$ provided that there is a probability measure absolutely continuous with respect to $\mathcal{H}^{s}$ satisfying certain density condition.

\medskip

The rest of this paper is organized as follows.

In the first section---\nameref{secPrelim}---we present the notation that will be used, collect the preliminaries on geometric measure theory and anisotropic (outer) Minkowski content, and lay out all the tools we need in the following sections, particularly \emph{sets of finite perimeter}, \emph{the Besicovitch covering theorem}, \emph{anisotropic perimeter}, \emph{anisotropic (outer) Minkowski content} and their properties.

In the second section---\nameref{secDef}---we introduce the notions of $s$-dimensional Minkowski content and $\mathcal{S}$-content. 

In the third section---\nameref{secDim}---we define the so-called Minkowski dimension and $\mathcal{S}$-dimension of a given compact set in $\mathbb{R}^n$. We also discuss the role of $C$ in these definitions. We show that the $C$-anisotropic volume function is Kneser and use this property to develop certain inequalities.

In the last section---\nameref{secFurther}---we give a proof of an existence-type theorem for the $s$-dimensional $C$-anisotropic Minkowski content and we take a closer look at the limit at zero from the right of the function $$r\mapsto\frac{V_{E, C}^{\prime}(r)}{r^{n-1}}.$$

\section{Preliminaries}
\label{secPrelim}
\subsection{Notation}
 Let $n\in\mathbb{N}$ (and so $n\geq1$). For $x, y\in\mathbb{R}^n$, we denote by $x\cdot y=\sum_{i=1}^{n}x_i y_i$ the \textit{Euclidean scalar product of $x$ and $y$}, and the symbol $|\cdot|$ stands for the norm induced by this scalar product.

If $A, B\subseteq\mathbb{R}^n$ are general and $E, F\subseteq\mathbb{R}^n$ are measurable, then: \begin{itemize}
    \item The symbol $A\oplus B\coloneq \{a+b;\;a\in A,\;b\in B\}$ stands for the \textit{Minkowski sum of A and B}. 
    \item The \textit{$r$-multiple of $A$} is denoted by $rA\coloneq\{ra;\;a\in A\}$, where $r\in\mathbb{R}$.
    \item The \textit{symmetric difference of $A$ and $B$} is denoted by $A\Delta B\coloneq (A\setminus B)\cup(B\setminus A)$. 
    \item We denote the \textit{interior of $A$} by $\inter (A),$ the \textit{exterior of $A$} by $\ext (A)$ and the \textit{boundary of $A$} 
    by $\partial A.$ 
    \item If $A\neq\emptyset$, then $\dist(\cdot, A)$ denotes the classical \textit{Euclidean distance from $A$} and $\diam(A)$ denotes the \textit{diameter of $A$.}
    \item The the symbol $\lambda^n(E)$ stands for the \textit{Lebesgue measure of $E$} and $\mathcal{H}^s(A)$ stands for \textit{$s$-dimensional Hausdorff measure of $A$}, where $s\geq0$, (note that for $s=0$ it is the \textit{counting measure}).
    \item The \textit{characteristic function of $A$} will be denoted by  $\chi_A$.
\end{itemize}

Further, let $B(x, r)$ denote the closed ball centered at $x$ with radius $r>0$. We set $\mathbb{S}^{n-1}\coloneq\partial B(0, 1).$

 Let $\mu$ be a nonnegative measure and let $B$ be a $\mu$-measurable set. The symbol $\mu\big|_{B}$ stands for the \textit{restriction of $\mu$ to $B$.} We abbreviate \textit{almost everywhere} by a.e. and \textit{almost all} by a.a. 

 The space $\mathcal{C}^{\infty}_c(\Omega)$ consists of all infinitely differentiable compactly supported functions on $\Omega$. 

 Let $\{E_{\varepsilon}\}_{\varepsilon>0}$ be a family of $\lambda^n$-measurable subsets of an open set $\Omega\subseteq\mathbb{R}^n.$ We say that $\{E_{\varepsilon}\}_{\varepsilon>0}$ \textit{converges towards a set $E$ in $\mathcal{L}^1_{\mathrm{loc}}(\Omega)$} if $\chi_{E_{\varepsilon}}\xrightarrow{\varepsilon\to0_{+}}\chi_{E}$ in $\mathcal{L}^1_{\mathrm{loc}}(\Omega)$, i.e. in the space of all locally integrable functions on $\Omega$.
     
    The symbol $\nabla$ will be used for the \textit{total derivative} (and we will identify the derivative and the corresponding matrix that represents the derivative).

 \subsection{Geometric Measure Theory}
In this subsection, we formulate some useful results from geometric measure theory. 
\begin{definition}
        Let $k\in\mathbb{N}_0$ with $k\leq n$ and $S\subseteq\mathbb{R}^n$ be $\mathcal{H}^k$-measurable. 
        
        The set $S$ is said to \textit{$k$-rectifiable} provided that there exists a bounded set $F\subseteq\mathbb{R}^k$ and a Lipschitz map $f\colon F\to\mathbb{R}^n$ such that $f(F)=S.$
        
        The set $S$ is said to be \textit{countably $\mathcal{H}^{k}$-rectifiable} if there exists a sequence $\{f_i\}_{i=0}^{\infty}$ of Lipschitz functions from $\mathbb{R}^k$ to $\mathbb{R}^n$ such that $$\mathcal{H}^k\left(E\setminus \bigcup_{i\in\mathbb{N}_0}f_i\left(\mathbb{R}^k\right)\right)=0.$$
 \end{definition}

 In fact, the definition of $k$-rectifiability of the set $S$ is equivalent to the fact that $S$ is contained in a finite union of Lipschitz surfaces.
\begin{rem}\label{remark}
    Given an $\mathcal{H}^{n-1}$-rectifiable set $S\subseteq\mathbb{R}^n$ with $\mathcal{H}^{n-1}(S)<\infty$, there exists a countable family of pairwise disjoint compact subsets of $S$ that covers $S$ up to an $\mathcal{H}^{n-1}$-negligible set and each element of which is a subset of a Lipschitz $(n-1)$-graph with finite $\mathcal{H}^{n-1}$ measure (see \cite[Lemma 11.1 and Remark 11.7]{simon}). 

\end{rem}

\begin{definition}
    Let $\Omega\subseteq\mathbb{R}^n$ be open and $E\subseteq\Omega$ be measurable. The set $E$ is said to be a \textit{set of finite perimeter in $\Omega$} if the derivative of $\chi_E$ in the sense of distributions, denoted by $D\chi_E,$ is an $\mathbb{R}^n$-valued Radon measure with finite total variation in $\Omega$, i.e. it is an $\mathbb{R}^n$-valued Radon measure with finite total variation in $\Omega$ satisfying $$\int\limits_{E}\nabla f\dif^{n}=-\int\limits_{\Omega}f\textup{ d}D\chi_E\quad\text{for all $f\in\mathcal{C}_c^{\infty}(\Omega)$}.$$ The \textit{perimeter of E in $\Omega$} is given by $$\Per(E; \Omega)\coloneq \Vert D\chi_E\Vert (\Omega),$$ i.e. the \textit{total variation of $D\chi_E$ in $\Omega$.}
\end{definition}

\begin{definition}
    Let $E\subseteq\mathbb{R}^n$ be measurable and $x\in\mathbb{R}^n.$ We define the \textit{upper} and \textit{lower Lebesgue density of E at x} as \begin{align*}
        \Theta_n(E, x)_*\coloneq \liminf_{r\to0_+}\frac{\lambda^n\Big|_{E}\big(B(x, r)\big)}{\lambda^n\big(B(x, r)\big)}&&\text{and}&&\Theta_n(E, x)^*\coloneq \limsup_{r\to0_+}\frac{\lambda^n\Big|_{E}\big(B(x, r)\big)}{\lambda^n\big(B(x, r)\big)}.
    \end{align*} Further, let $\Theta_n(E, x)$ be their common value (i.e. $\Theta_n(E, x)\coloneq \Theta_n(E, x)_*$ provided that $\Theta_n(E, x)_*=\Theta_n(E, x)^*$) which is referred to as the \textit{Lebesgue density of E at x}.

    Denote $$E^t\coloneq \{x\in\mathbb{R}^n;\; \Theta_n(E, x)=t\}\quad\text{for $t\in [0, 1]$}.$$ The \textit{essential boundary of E} is given by $$\partial_*E\coloneq \mathbb{R}^n\setminus\left(E^{0}\cup E^1\right).$$ Let us also introduce the \textit{reduced boundary of E} defined by $$\partial^*E\coloneq \left\{x\in E^{1/2};\; \exists \nu_E(x)\in\mathbb{S}^{n-1}:\; \frac{E-x}{\varepsilon}\xrightarrow{\varepsilon\to0_{+}}\{y\in\mathbb{R}^n;\;y\cdot\nu_E(x)\leq0\}\right\},$$ where we consider the $\mathcal{L}^1_{\mathrm{loc}}(\mathbb{R}^n)$-convergence of sets. Moreover, if $x\in\partial^* E,$ we call $\nu_E(x)$ the \textit{outer unit normal to E at x.}  
\end{definition}
\begin{rem}\label{remarko}
    Let us state the following facts (see \cite{ambrosio}, section 3.3).\begin{itemize}
 \item For a set $E$ of finite perimeter in $\Omega$, we always have $$\Per(E; \Omega)=\mathcal{H}^{n-1}\left(\partial_* E\cap \Omega\right)=\mathcal{H}^{n-1}\left(\partial^* E\cap \Omega\right).$$
 \item We have $\Per(E; \Omega)=\Per(F; \Omega)$ whenever $\lambda^n\big((E\Delta F)\cap\Omega\big)=0.$
\item If $E$ is a set of finite perimeter in $\Omega,$ then $\Omega\setminus E$ is also a set of finite perimeter in $\Omega$ with $\nu_E=-\nu_{\Omega\setminus E}$ $\left(\mathcal{H}^{n-1}\Big|_{\partial^* E\cap \Omega}\right)$-a.e., $\mathcal{H}^{n-1}(\partial^* E\cap\Omega) = \mathcal{H}^{n-1}\left(\partial^* \left(\Omega\setminus E\right)\cap\Omega\right)$ and $\textup{Per}(E; \Omega)=\textup{Per}\big(\Omega\setminus E; \Omega\big).$

    \end{itemize}
\end{rem}

Let us now recall the Besicovitch covering theorem (see \cite[subsection 1.5.2, Theorem 2]{evans}):
\begin{theorem}\label{bes}
       Let $A\subseteq\mathbb{R}^n$ and $\rho\colon A\to (0, \infty)$ be a bounded function. Then there exist $N\in\mathbb{N}$ depending only on $n$ and $S\subseteq A$ at most countable such that \begin{align*}
          A\subseteq\bigcup_{x\in S}B\big(x, \rho(x)\big) &&\text{and}&& \sum_{x\in S}\chi_{B(x, \rho(x))}\leq N.
       \end{align*}
\end{theorem}

\subsection{Anisotropic (Outer) Minkowski Content}
In this subsection, we will summarize the basic facts about anisotropic (outer) Minkowski content from \cite{chambolle} and \cite{Villa2}.
\begin{definition}\label{convexity}
    A \textit{convex body} is a non-empty, compact and convex subset of $\mathbb{R}^n$. The system of all convex bodies the interior of which contains the origin will be denoted by the symbol $\mathcal{C}^n_0$ and the system of all compact subsets of $\mathbb{R}^n$ will be denoted by $\mathcal{K}^n.$
    
    Let $C\in\mathcal{C}^n_0$. The \textit{support function of $C$} is given by the formula $$h_C(y)\coloneq \sup_{x\in C}x\cdot y\quad \text{for $y\in\mathbb{R}^n.$}$$ 
     We also introduce the \textit{polar function of C} given by $$h^{\circ}_C(x)\coloneq \sup_{\substack{y\in\mathbb{R}^n\\ h_C(y)\leq1}} x\cdot y\quad \text{for $x\in\mathbb{R}^n.$}$$ For $E\subseteq\mathbb{R}^n$ with $E\neq\emptyset$, we define the \textit{$C$-anisotropic distance function to the set $E$} by $$\textup{dist}_{C}(x, E)\coloneq \underset{y\in E}{\textup{inf}} \textit{ h}^{\circ}_C(x-y)\quad\text{for $x\in\mathbb{R}^n.$}$$ 
\end{definition}

\begin{rem} Let $C\in\mathcal{C}^n_0$.
    \begin{itemize}
    \item We define the \textit{polar of $C$} by $C^{\circ}\coloneq\{x\in\mathbb{R}^n;\;h_C(x)\leq 1\}$. We see that $h^{\circ}_C=h_{C^{\circ}}$ and $(C^{\circ})^{\circ}=C$.
        \item The function $x\mapsto h_C(x)$ is sublinear, hence convex, and Lipschitz with the property $h_{-C}(x)=h_{C}(-x)$ for all $x\in\mathbb{R}^n$. It is also true that $h_{aC}=ah_{C}$ for any $a>0$.
 \end{itemize}\end{rem}

\begin{definition}\label{outerM}
    Let $Q\subseteq\mathbb{R}^n$ be with $0\in Q$, $\Omega\subseteq\mathbb{R}^n$ be open and $\varepsilon>0$. We define the $(\varepsilon, Q)$-\textit{anisotropic outer Minkowski content (in $\Omega$)} as a functional defined on $\lambda^n$-measurable subsets $E\subseteq\mathbb{R}^n$ by the formula $$\mathcal{SM}_{\varepsilon, Q}(E; \Omega)\coloneq \frac{1}{\varepsilon}\lambda^n\Big(\big((E\cap\Omega)\oplus\varepsilon Q\big)\cap(\Omega\setminus E)\Big).$$ Further, let \begin{align*}
        \mathcal{SM}_{Q}(E; \Omega)_{*}\coloneq \liminf_{\varepsilon\to0_{+}}\mathcal{SM}_{\varepsilon, Q}(E; \Omega)&&\text{and}&&\mathcal{SM}_{Q}(E; \Omega)^{*}\coloneq \limsup_{\varepsilon\to0_{+}}\mathcal{SM}_{\varepsilon, Q}(E; \Omega).
    \end{align*}  The former formula defines the \textit{lower $Q$-anisotropic outer Minkowski content of $E$ (in $\Omega$)}, the latter the \textit{upper $Q$-anisotropic outer Minkowski content of $E$ (in $\Omega$)}. If $\mathcal{SM}_{Q}(E; \Omega)_*=\mathcal{SM}_{Q}(E; \Omega)^*$, we define the \textit{$Q$-anisotropic outer Minkowski content (in $\Omega$)} by the formula $\mathcal{SM}_{Q}(E; \Omega)\coloneq\mathcal{SM}_{Q}(E; \Omega)_*$. If we write $\mathcal{SM}_{Q}(E; \Omega)$, we implicitly mean that the limit exists.
\end{definition}
    
    Let us define the so-called $C$-anisotropic Minkowski content, which is useful for $\lambda^n$-negligible sets (unlike $\mathcal{SM}_C$):
    \begin{definition}\label{Minkow}
        Let $S\subseteq\mathbb{R}^n$ be measurable and $Q\subseteq\mathbb{R}^n$ with $0\in Q$. Putting $$\mathcal{M}_{\varepsilon, Q}(S; \Omega)\coloneq\frac{1}{2{\varepsilon}}\lambda^n\Big(\big((S\cap\Omega)\oplus\varepsilon Q\big)\cap\Omega\Big),\quad\text{$\varepsilon>0$},$$ we define the \textit{$(\varepsilon, Q)$-anisotropic Minkowski content of $S$ in $\Omega$.} We define the \textit{lower} and \textit{upper $Q$-anisotropic Minkowski content of $S$ in $\Omega$} by \begin{align*}
            \mathcal{M}_{Q}(S; \Omega)_*\coloneq\liminf_{\varepsilon\to0_+}\mathcal{M}_{\varepsilon, Q}(S; \Omega)&&\text{and}&&\mathcal{M}_{Q}(S; \Omega)^*\coloneq\limsup_{\varepsilon\to0_+}\mathcal{M}_{\varepsilon, Q}(S; \Omega). 
        \end{align*} If $\mathcal{M}_{Q}(S; \Omega)_*=\mathcal{M}_{Q}(S; \Omega)^*$, we define the \textit{$Q$-anisotropic Minkowski content of $S$ in $\Omega$} by $\mathcal{M}_{Q}(S; \Omega)\coloneq\mathcal{M}_{Q}(S; \Omega)_*$. If we write $\mathcal{M}_{Q}(S; \Omega)$, we implicitly mean that the limit exists.
    \end{definition}

    In this work, we provide a generalization of the notions introduced in Definitions \ref{outerM} and \ref{Minkow} for $\Omega=\mathbb{R}^n$. Namely, we replace the $\varepsilon$ in the denominators by $\varepsilon^{n-s}$, where $s\in[0, n].$ Roughly speaking, the number $s$ reflects in some sense the ``dimension'' of the set studied if the limit is neither zero nor infinity.

The meaning of the word \textit{isotropic} is ``having the same properties in all directions'', while the word \textit{anisotropic} is its opposite, i.e. ``not having the same properties in all directions'' (see \href{https://www.etymonline.com/word/isotropic}{the etymology of (an)isotropic}).    

\begin{definition}
    Let $C\in\mathcal{C}^n_0$ and $\Omega\subseteq\mathbb{R}^n$ be open. For $E\subseteq\mathbb{R}^n$ measurable we define the \textit{$C$-anisotropic perimeter of $E$ in $\Omega$} by \[
     \textup{Per}_{h_{C}}(E; \Omega)\coloneq \begin{cases}
         {\displaystyle \int\limits_{\Omega\cap\partial^* E}}h_C\left(\nu_E\right)\difm^{n-1
    },& \text{  if $\Per(E; \Omega)<\infty$},\\+\infty, & \text{  otherwise}.
     \end{cases}  
\]
\end{definition}

We have a sufficient condition for the existence of the limit $\mathcal{SM}_{C}(E; \mathbb{R}^n)$:
\begin{theorem}[\protect{\cite[Theorem 4.3]{Villa2}}]\label{closed}
    Let $C\in\mathcal{C}^n_0$ and $E$ be a closed subset of $\mathbb{R}^n$ such that $\partial E$ is a countably $\mathcal{H}^{n-1}$-rectifiable and bounded set and such that there exists $\gamma>0$ and a probability measure $\eta$ absolutely continuous with respect to $\mathcal{H}^{n-1}$ satisfying $$\forall r>0\; \forall x\in\partial E: \eta\big(B(x, r)\big)\geq \gamma r^{n-1}.$$ Then we have $$\mathcal{SM}_{C}(E; \mathbb{R}^n)=\textup{Per}_{h_C}(E; \mathbb{R}^n)+\int\limits_{E^{0}\cap\partial E}\big(h_C(\nu_E)+h_C(-\nu_E)\big)\difm^{n-1}.$$
\end{theorem}

Let us introduce the \textit{anisotropic isoperimetric inequality} (see \cite{Fonseca}, section 4). 
\begin{theorem}\label{asas}
    Let $E\subseteq\mathbb{R}^n$ be a set of finite perimeter in $\mathbb{R}^n$ and of finite Lebesgue measure, and $C\in\mathcal{C}^n_0$. Then $$\textup{Per}_{h_{C}}(E; \mathbb{R}^n)\geq n\lambda^n(C)^{\frac{1}{n}}\lambda^n(E)^{\frac{n-1}{n}}.$$
\end{theorem}

\section{Definitions of $s$-dimensional $C$-anisotropic (Outer) Minkowski Content and $\mathcal{S}$-content}
\label{secDef}
Let us first introduce the notion of \textit{anisotropic tubular neighborhoods.} This notion is further developed in \cite{Villa} and \cite{Rataj}, where we refer the reader for a more detailed exposition.
\begin{definition}
    Let $E, C\subseteq\mathbb{R}^n$ be arbitrary subsets and $r \geq0.$ If $0\in C$, we set $$E_{r, C}\coloneq  E\oplus rC.$$ The set $E_{r, C}$ is called the \textit{$(r, C)$-anisotropic tubular neighbourhood of E.} Moreover, if $C$ is convex, we define for $r>0$ $$E^{\prime}_{r, C}\coloneq \bigcup_{0<s<r}(E\oplus sC).$$ Further, we set $$V_{E, C}(r)\coloneq \mathcal{H}^{n}(E_{r, C}).$$ We will refer to the function $V_{E, C}\colon[0, \infty)\to[0, \infty]$ as the \textit{C-anisotropic volume function of E} in the sequel.
\end{definition}

     In our setting, the set $E$ will be compact most of the time and $C$ will be a convex body of full dimension with $0\in\textup{int}(C)$. In these cases, the sets $\partial E_{r, C}$ and $\partial E^{\prime}_{r, C}$ are finite unions of Lipschitz surfaces \cite[Theorem 3.1]{Villa} for all $r>0$.
    
    We will summarize the properties of the function $V_{E, C}(\cdot)$. The non-obvious ones are proven in \cite[section 5]{Villa}:

Let $E\in\mathcal{K}^n$ and $C\in\mathcal{C}^n_0$. It is plain to see that the function $V_{E, C}(\cdot)$ is increasing and continuous from $[0, \infty)$ to $[0, \infty)$ (as follows from the continuity of the measure $\lambda^n$). Moreover, the left and right hand side derivatives of $V_{E, C}(\cdot)$, denoted by $V_{E, C}^-(\cdot)$ and $V_{E, C}^+(\cdot)$ respectively, exist at each point of $(0, \infty)$; $V_{E, C}^-(\cdot)\geq V_{E, C}^+(\cdot)$ and are continuous from the left and from the right, respectively. Finally, they are given by the formul\ae{} $$V_{E, C}^-(r)=\int\limits_{\partial^* (E_{r, C})}h_C(\nu_{E_{r, C}})\difm^{n-1}+\int\limits_{(E_{r, C})^1\cap \partial (E^{\prime}_{r, C})}\big(h_C(\nu_{E^{\prime}_{r, C}})+h_C(-\nu_{E^{\prime}_{r, C}})\big)\difm^{n-1}$$ and $$V_{E, C}^+(r)=\int\limits_{\partial^* (E_{r, C})}h_C(\nu_{E_{r, C}})\difm^{n-1}.$$ In particular, the function $V_{E, C}(\cdot)$ is first of all of class $\mathcal{C}^1$ on $(0, \infty)\setminus J_{E, C}$, where $J_{E, C}$ is an at most countable set given by $$J_{E, C}\coloneq \left\{r>0;\; \mathcal{H}^{n-1}\big((E_{r, C})^1\cap\partial (E^{\prime}_{r, C})\big)>0\right\},$$ and, secondly, is differentiable at a point $r>0$ if and only if $r\notin J_{E, C}$. In this case, \begin{equation}\label{juj}
    (V_{E, C})^{\prime}(r)=\int\limits_{\partial^* (E_{r, C})}h_{C}\left(\nu_{E_{r, C}}\right)\difm^{n-1}=\textup{Per}_{h_{C}}(E_{r, C}; \mathbb{R}^n).
\end{equation}

From the aforementioned properties, it easily follows that $V_{E, C}(\cdot)$ is \textit{locally absolutely continuous on $[0, \infty)$}, that is, absolutely continuous on each compact interval contained in $[0, \infty)$ (see \cite{Heil}, Theorem 17 and Corollary 20). Formula (\ref{juj}) inspires us to introduce the following definition:
\begin{definition}
    Let $C\in\mathcal{C}^n_0$ and $E\in\mathcal{K}^n$. We define$$S_{E, C}(r)\coloneq \textup{Per}_{h_{C}}(E_{r, C}; \mathbb{R}^n),\;r>0.$$
\end{definition}   

\begin{rem}
     Let $E\in\mathcal{K}^n$ and $C\in\mathcal{C}^n_0$. \begin{itemize}
         \item We always have $E_{r, C}=\{x\in\mathbb{R}^n;\; \textup{dist}_{C}(x, E)\leq r\}$ and $E^{\prime}_{r, C}=\{x\in\mathbb{R}^n;\; \textup{dist}_{C}(x, E)< r\}$ for any $r>0$ and nonempty $E$.
     \item We have $S_{E, C}(r)=V_{E, C}^+(r)$ for any $r>0$. In particular, the function $S_{E, C}(\cdot)$ is continuous from the right.
     \end{itemize} 
\end{rem}\
\begin{lemma}\label{L3}
    For any $C, C^{\prime}\in\mathcal{C}^n_0$, $E\in\mathcal{K}^n$ and $r>0$, we have $$\mathcal{SM}_{C^{\prime}}(E_{r, C}; \mathbb{R}^n)=\textup{Per}_{h_{C^{\prime}}}(E_{r, C}; \mathbb{R}^n).$$ In particular, $$\mathcal{S}_{E, C}(r)=\mathcal{SM}_{C}(E_{r, C}; \mathbb{R}^n).$$
    \begin{proof}
        Since the set $E_{r, C}$ satisfies all assumptions of Theorem \ref{closed}, we deduce that $$\mathcal{SM}_{C^{\prime}}(E_{r, C}; \mathbb{R}^n)=\textup{Per}_{h_{C^{\prime}}}(E_{r, C}; \mathbb{R}^n)+\int\limits_{(E_{r, C})^0\cap\partial(E_{r, C})}\big(h_{C^{\prime}}(\nu_{E_{r, C}})+h_{C^{\prime}}(-\nu_{E_{r, C}})\big)\difm^{n-1}.$$ We will show that $(E_{r, C})^0\cap\partial(E_{r, C})=\emptyset,$ which in turn implies $$\mathcal{SM}_{C^{\prime}}(E_{r, C}; \mathbb{R}^n)=\textup{Per}_{h_{C^{\prime}}}(E_{r, C}; \mathbb{R}^n).$$Given $x_0\in (E_{r, C})^0\cap\partial(E_{r, C})$, we have on the one hand that there exists $y\in E$ so that $$x_0\in y+rC$$ and, on the other hand,  $$\lim_{\varepsilon\to0_+}\frac{\lambda^n\big(E_{r, C}\cap B(x_0, \varepsilon)\big)}{\varepsilon^n}=0,$$ whence $$\lim_{\varepsilon\to0_+}\frac{\lambda^n\big((y+rC)\cap B(x_0, \varepsilon)\big)}{\varepsilon^n}=0,$$ which is a contradiction since each point of the convex body $y+rC$ cannot have density zero.
    \end{proof}
\end{lemma}
Let us introduce \textit{$s$-dimensional anisotropic Minkowski content}. Let us point out that it is a natural extension of Definition \ref{Minkow}.
\begin{definition}\label{def1} Let $C\in\mathcal{C}^n_0$,  $E\subseteq\mathbb{R}^n$ and $s\in [0, n]$. Putting $$\mathcal{M}^{s}_{r, C}(E)\coloneq \frac{V_{E, C}(r)}{r^{n-s}},\;r>0,$$ we define the \textit{$s$-dimensional lower} and \textit{upper $C$-anisotropic Minkowski content of $E$} by \begin{align*}
    \mathcal{M}^s_{C}(E)_*\coloneq \liminf_{r\to0_+}\mathcal{M}^{s}_{r, C}(E)&& \text{and}&&\mathcal{M}^s_{C}(E)^*\coloneq \limsup_{r\to0_+}\mathcal{M}^{s}_{r, C}(E).
\end{align*} Whenever the two quantities are equal, that is, $\mathcal{M}^s_{C}(E)_*=\mathcal{M}^s_{C}(E)^*$, we define the \textit{$s$-dimensional $C$-anisotropic Minkowski content of $E$} by $$\mathcal{M}^s_{C}(E)\coloneq \mathcal{M}^s_{C}(E)^*.$$
\end{definition}
The following lemma shows that the previous definition is of interest only for $s\in (0, n).$
\begin{lemma}\label{inter}
    Let $E\in\mathcal{K}^n$ and $C\in\mathcal{C}^n_0$. Then
        \begin{align*}
            \mathcal{M}^n_{C}(E)=\lambda^n(E)&& \text{and} &&
        \mathcal{M}^0_{C}(E)=\mathcal{H}^0(E)\lambda^n(C).
        \end{align*}
    
    \begin{proof}
      The first assertion follows easily, as
        $$\mathcal{M}^n_{C}(E)_*=\mathcal{M}^n_{C}(E)^*=\lim_{r\to0_+}V_{E, C}(r)=V_{E, C}(0)=\lambda^n(E).$$ To prove the second part, suppose first that $\mathcal{H}^0(E)=k\in\mathbb{N}_0$. We see that in this case $V_{E, C}(r)=kr^n\lambda^n(C)$ for small enough $r$. Hence $$\mathcal{M}^0_{C}(E)_*=\mathcal{M}^0_{C}(E)^*=\lim_{r\to0_+}\frac{V_{E, C}(r)}{r^n}=k\lambda^n(C)=\mathcal{H}^0(E)\lambda^n(C).$$ If, on the other hand, $\mathcal{H}^0(E)=\infty$, we can find for each $k\in\mathbb{N}$ a subset $E_k\subseteq E$ with $\mathcal{H}^0(E_k)=k.$ But then  $$\liminf_{r\to0_+}\frac{V_{E, C}(r)}{r^n}\geq\liminf_{r\to0_+}\frac{V_{E_k, C}(r)}{r^n}=k\lambda^n(C) \text{ for each $k\in\mathbb{N}$}.$$ Hence $\mathcal{M}^0_{C}(E)_*=\mathcal{M}^0_{C}(E)^*=\infty.$
    \end{proof}
\end{lemma}

 It is easy to check that $s$-dimensional (anisotropic) Minkowski content is interesting only for $\lambda^n$-negligible sets (since for $s\in[0, n)$ we have the estimate $\mathcal{M}^s_{r, C}(E)\geq \lambda^n(E)r^{s-n}$, from which it follows $\mathcal{M}^s_{C}(E)=\infty$ provided that $\lambda^n(E)>0$). However, we can investigate the convergence of $$\frac{V_{E, C}(r)-V_{E, C}(0)}{r^{n-s}}$$ as $r$ tends to zero from the right. Notice that for $s=n-1$ we obtain precisely $(r, C)$-anisotropic outer Minkowski content of $E$ introduced in the section \nameref{secPrelim}.

 \begin{definition}\label{def2}
     Let $K\in\mathcal{K}^n$, $C\in\mathcal{C}^n_0$ and $s\in[0, n]$. Putting $$\mathcal{SM}^s_{r, C}(E)\coloneq \frac{V_{E, C}(r)-V_{E, C}(0)}{r^{n-s}},\;r>0,$$ we define the \textit{$s$-dimensional lower} and \textit{upper $C$-anisotropic outer Minkowski content of $E$} by \begin{align*}
         \mathcal{SM}^s_{C}(E)_*\coloneq \liminf_{r\to0_+}\mathcal{SM}^{s}_{r, C}(E)&&\text{and}&&\mathcal{SM}^s_{C}(E)^*\coloneq \limsup_{r\to0_+}\mathcal{SM}^{s}_{r, C}(E).
     \end{align*} Whenever the two values are equal, that is, $\mathcal{SM}^s_{C}(E)_*=\mathcal{SM}^s_{C}(E)^*$, we define the \textit{s-dimensional $C$-anisotropic outer Minkowski content of $E$} by $$\mathcal{SM}^s_{C}(E)\coloneq \mathcal{SM}^s_{C}(E)^*.$$
 \end{definition}
\begin{rem} Let us make the following comments.\begin{itemize}
    \item We see that Definition \ref{def1} and Definition \ref{def2} are the same for $E$ with zero Lebesgue measure.
    \item Observe that for open $\Omega\subseteq\mathbb{R}^n$, $E\in\mathcal{K}^n$ with $E\subseteq\Omega$, and $C\in\mathcal{C}^n_0$ we have $\mathcal{SM}_{r, C}(E; \Omega)=\mathcal{SM}^{n-1}_{r, C}(E)$ and $2\,\mathcal{M}_{r, C}(E; \Omega)=\mathcal{M}^{n-1}_{r, C}(E)$ for any $r>0$.
\end{itemize}
\end{rem}
 
\begin{definition}\label{def3} Let $E\in\mathcal{K}^n$, $C\in\mathcal{C}^n_0$ and $s\in [0, n]$. We define the \textit{s-dimensional lower} and \textit{upper C-anisotropic $\mathcal{S}$-content of $E$} for $s<n$ by \begin{align*}
    \mathcal{S}^s_{C}(E)_*\coloneq \liminf_{r\to0_+}\frac{S_{E, C}(r)}{(n-s)r^{n-s-1}}&& \text{and} &&\mathcal{S}^s_{C}(E)^*\coloneq \limsup_{r\to0_+}\frac{S_{E, C}(r)}{(n-s)r^{n-s-1}},
\end{align*} and for $s=n$ we simply set $\mathcal{S}^n_{C}(E)_*=\mathcal{S}^n_{C}(E)^*\coloneq 0.$ Whenever the both values are equal, that is, $\mathcal{S}^s_{C}(E)_*=\mathcal{S}^s_{C}(E)^*$, we define the \textit{s-dimensional C-anisotropic $\mathcal{S}$-content of $E$} by $$\mathcal{S}^s_{C}(E)\coloneq \mathcal{S}^s_{C}(E)^*.$$
\end{definition}
\begin{rem}
  It is possible to ``normalize'' the contents introduced in Definitions \ref{def1}, \ref{def2} and \ref{def3} by multiplying the fractions by $\omega_{n-s}^{-1}$. This ``normalization'' makes sense only in the case of isotropic contents, i.e. $C=B(0, 1)$ (see \cite{Rataj}).
\end{rem}

\section{Minkowski {D}imension and $\mathcal{S}$-dimension}
\label{secDim}
Let us start this section with the following definition.
\begin{definition}
    Let $f: [0, \infty)\to [0, \infty)$ be a continuous function and $n\in\mathbb{N}$. The function $f$ is a \textit{Kneser function} (or \textit{of Kneser type}, respectively) \textit{of order $n$} provided that it satisfies $$f(t b) - f(t a)\leq t^n\big(f(b)-f(a)\big)\quad\text{for any $0<a\leq b$ and any $t\geq 1$.}$$ 
\end{definition}

Kneser functions have the following properties (cf. \cite{Stachó}, Lemma 2). Interestingly, they share these properties with $V_{E, C}(\cdot)$.
\begin{lemma}\label{kneser}
    Any Kneser function $f:[0, \infty)\to[0, \infty)$ is locally absolutely continuous; $f^{\prime}$ exists everywhere, safe for an at most countable set $J_f\subseteq(0, \infty)$; the left and right hand side derivatives of $f$ (denoted by $f^-$ and $f^+$, respectively) exist at each point of $(0, \infty)$; $f^-\geq f^+$; and $f^-$ and $f^+$ are continuous from the left and from the right, respectively.
\end{lemma}

A natural question arises: Is the function $V_{E, C}(\cdot)$ of Kneser type of order $n$ for any convex body of full dimension $C\subseteq\mathbb{R}^n$ with $0\in\textup{int}(C)$, and any compact set $E\subseteq\mathbb{R}^n$? A well-known result of Kneser (see \cite{Kneser}) gives an affirmative answer to this question for $C=B(0, 1)$. Few years later, Stachó showed that the volume function is of Kneser type for any $C$ that is open, bounded, central symmetric, and convex (see \cite{Stachó}). 

We will give a proof for a general $C$. But first, let us present Corollary 5.5 from \cite{Villa}:
\begin{lemma}\label{mnm}
    Let $E\in\mathcal{K}^n$, $C\in\mathcal{C}^n_0$, and set $$\kappa_{E, C}(r)\coloneq \frac{1}{r^{n-1}}\int\limits_{\partial^* (E_{r, C})}h_C\left(\nu_{E_{r, C}}\right)\difm^{n-1},\;r>0.$$ Then for any pair $r_0, r_1\in (0, \infty)\setminus J_{E, C}$ with $r_0\leq r_1$ we have $\kappa_{E, C}(r_0)\geq\kappa_{E, C}(r_1)$. Moreover, if $E\neq\emptyset$, we have $$\lim_{r\to\infty}\kappa_{E, C}(r)=n\lambda^n(C).$$
\end{lemma}
It is interesting that the limit at infinity is the same for all non-empty compacta $E$. It is quite natural to ask what the limit $\lim_{r\to0_+}\kappa_{E, C}(r)$ is, and we will compute it at the end of this chapter. But let us now return to our study of $V_{E, C}$.
\begin{theorem}
    Let $E\in\mathcal{K}^n$. Then the function $V_{E, C}(\cdot)$ is of Kneser type of order $n$ for any $C\in\mathcal{C}^n_0$.
\end{theorem}
\begin{proof}
    Fix $b>a>0$. We already know that $V_{E, C}(\cdot)$ is locally absolutely continuous on the interval $[0, \infty)$. Hence, for any $t\geq1$ we have\begin{equation*}
        \begin{split}
            V_{E, C}(t b)-V_{E, C}(t a)&=\int\limits_{t a}^{t b}V_{E, C}^{\prime}(r)\dif^1(r) \\&=\int\limits_{t a}^{t b}\int\limits_{\partial^* (E_{r, C})}h_C\left(\nu_{E_{r, C}}(x)\right)\difm^{n-1}(x)\dif^1(r)\\&=\int\limits_{t a}^{t b}r^{n-1}\left(\frac{1}{r^{n-1}}\int\limits_{\partial^* (E_{r, C})}h_C\left(\nu_{E_{r, C}}(x)\right)\difm^{n-1}(x)\right)\dif^1(r)\\&=\int\limits_{t a}^{t b}r^{n-1}\kappa_{E, C}(r)\dif^1(r)=\int\limits_{a}^{b}(t\rho)^{n-1}\kappa_{E, C}(t \rho) t\dif^1(\rho)\\&=t^n\int\limits_{a}^{b}\rho^{n-1}\kappa_{E, C}(t \rho) \dif^1(\rho)\\&\leq t^n\int\limits_{a}^{b}\rho^{n-1}\kappa_{E, C}(\rho) \dif^1(\rho)= t^n\big(V_{E, C}(b)-V_{E, C}(a)\big),
        \end{split}
    \end{equation*}
    where the inequality is a straightforward consequence of Lemma \ref{mnm}.
\end{proof}

 Observe that if $E\in\mathcal{K}^n$ with zero Lebesgue measure and $C\in\mathcal{C}^n_0$ with $\mathcal{M}^s_{C}(E)^*\in(0, \infty)$ for some $s\in[0, n]$, then we have $\mathcal{M}^s_{C}(E)^*=0$ for any $t\in (s, n]$ and similarly $\mathcal{M}^s_{C}(E)^*=\infty$ for any $t\in[0, s)$. The same observation holds for lower anisotropic Minkowski content and for upper and lower anisotropic $\mathcal{S}$-content. This idea leads to the following definition.
\begin{definition}
    Let $E\in\mathcal{K}^n$ be $\lambda^n$-negligible, $C\in\mathcal{C}^n_0$ and $s\in [0, n]$. We define the \textit{s-dimensional lower} and \textit{upper $C$-anisotropic Minkowski dimension} by $$\dim_{\mathcal{M}}^C(E)_*\coloneq \inf\left\{t\in[0, n],\;\mathcal{M}^t_{C}(E)_*=0 \right\}$$ and $$\dim_{\mathcal{M}}^C(E)^*\coloneq \inf\left\{t\in[0, n],\;\mathcal{M}^t_{C}(E)^*=0 \right\},$$ respectively. If, moreover, the equality $\dim^C_{\mathcal{M}}(E)^*=\dim^C_{\mathcal{M}}(E)_*$  holds, we define the \textit{$C$-anisotropic Minkowski dimension of $E$} by $\dim^C_{\mathcal{M}}(E)\coloneq \dim^C_{\mathcal{M}}(E)^*$.
    
    Similarly, let us define the \textit{s-dimensional lower} and \textit{upper $C$-anisotropic $\mathcal{S}$-dimension} by 
    $$\dim^C_{\mathcal{S}}(E)_*\coloneq \inf\left\{t\in[0, n],\;\mathcal{S}^t_{C}(E)_*=0 \right\}$$ and $$\dim^C_{\mathcal{S}}(E)^*\coloneq \inf\left\{t\in[0, n],\;\mathcal{S}^t_{C}(E)^*=0 \right\},$$ respectively. If, moreover, the equality $\dim^C_{\mathcal{S}}(E)^*=\dim^C_{\mathcal{S}}(E)_*$ holds, we define the \textit{$C$-anisotropic $\mathcal{S}$-dimension of $E$} by $\dim^C_{\mathcal{S}}(E)\coloneq \dim^C_{\mathcal{S}}(E)^*$.
\end{definition}
\begin{rem}
    First, note that the notion of $C$-anisotropic Minkowski dimension does \emph{not} depend on the choice of $C$. Indeed, if $C^{\prime}\in\mathcal{C}^n_0$, we can find positive constants $a$ and $b$ so that $aC\subseteq C^{\prime} \subseteq bC$. Hence, for $E\in\mathcal{K}^n$ with zero measure we have $$a^{n-s}\mathcal{M}^{s}_{ar, C}(E)\leq\mathcal{M}^{s}_{r, C^{\prime}}(E)\leq b^{n-s}\mathcal{M}^{s}_{br, C}(E)\quad\text{for any $r>0.$}$$ So, we will write simply $\dim_{\mathcal{M}}$ instead of $\dim^C_{\mathcal{M}}$ and the modifier \emph{$C$-anisotropic} will be omitted. The question is whether the notion of $C$-anisotropic $\mathcal{S}$-dimension depends on $C$ or not. We will see that the notion of upper $C$-anisotropic $\mathcal{S}$-dimension does not depend on $C$, which in turn implies that if $\dim^C_{\mathcal{S}}(E)$ and $\dim^{C^{\prime}}_{\mathcal{S}}(E)$ exist for some $C, C^{\prime}\in\mathcal{C}^n_0$, then the dimensions must, in fact, be equal. 
\end{rem}
Let us prove an analogue of Proposition 3.1 from \cite{Rataj} (including Corollary 3.2). Its proof is a slight modification of that in \cite{Rataj}.
\begin{lemma}\label{lemma16}
    Let $f\colon[0, \infty)\to[0, \infty)$ be a Kneser function of order $n$ and let $h\colon[0, \infty)\to[0, \infty)$ be a continuously differentiable function with $h(0)=0$. Assume that $h^{\prime}$ is non-zero in some neighborhood on the right of $0$. Let \begin{align*}
        S_*\coloneq \liminf_{r\to0_+}\frac{f^{+}(r)}{h^{\prime}(r)}&& \text{and} &&S^*\coloneq \limsup_{r\to0_+}\frac{f^{+}(r)}{h^{\prime}(r)}.
    \end{align*} Then \begin{equation}\label{equa}
        S_*\leq\liminf_{r\to0_+}\frac{f(r)-f(0)}{h(r)}\leq\limsup_{r\to0_+}\frac{f(r)-f(0)}{h(r)}\leq S^*.
    \end{equation} In particular, if $E\in\mathcal{K}^n$, $C\in\mathcal{C}^n_0$, and $s\in[0, n]$, we get the following chain of inequalities: $$\mathcal{S}^s_{C}(E)_*\leq\mathcal{SM}^s_{C}(E)_*\leq\mathcal{SM}^s_{C}(E)^*\leq\mathcal{S}^s_{C}(E)^*.$$ If, moreover, $\lambda^n(E)=0$, then $$\mathcal{S}^s_{C}(E)_*\leq\mathcal{M}^s_{C}(E)_*\leq\mathcal{M}^s_{C}(E)^*\leq\mathcal{S}^s_{C}(E)^*.$$
\end{lemma}
\begin{proof} First, observe that the case $s=n$ is trivial. So, we will suppose that $s<n$.

    We claim that for each $r>0$ there exist $0< t_1(r), t_2(r) <r$ so that \begin{equation}\label{ukuk}
        \frac{f^{+}\big(t_1(r)\big)}{h^{\prime}\big(t_1(r)\big)}\leq\frac{f(r)-f(0)}{h(r)}\leq\frac{f^{+}\big(t_2(r)\big)}{h^{\prime}\big(t_2(r)\big)}.
    \end{equation} To prove this, fix a positive number $r$ and define $$\Phi(t)\coloneq \big(f(r)-f(0)\big)h(t)-h(r)f(t),\;t\in[0, r].$$ The function $\Phi(\cdot)$ is absolutely continuous on the interval $[0, r]$ and $\Phi(r)=\Phi(0).$ It follows $$\int\limits_{0}^{r}\Phi^{\prime}(t)\dif^1(t)=\Phi(r)-\Phi(0)=0.$$ Consequently, either $\Phi^{\prime}=0$ $\lambda^1$-a.e. in $(0, r)$, or there exist $t_1(r), t_2(r)$ in $(0, r)$ with $\Phi^{\prime}\big(t_1(r)\big)>0>\Phi^{\prime}\big(t_2(r)\big)$, from which we conclude (\ref{ukuk}). 
    
    Now $$\limsup_{r\to0_+}\frac{f(r)-f(0)}{h(r)}\leq \limsup_{r\to0_+}\frac{f^{+}\big(t_2(r)\big)}{h^{\prime}\big(t_2(r)\big)}\leq S^*, $$ and similarly for the other inequality.\medskip

    To prove the second part, set $h(t)\coloneq t^{n-s}$ and $f(t)\coloneq V_{E, C}(t)$ for $t\geq0$. We obtain \begin{equation*}
        \mathcal{S}^s_{C}(E)_*=\liminf_{r\to0_+}\frac{S_{E, C}(r)}{(n-s)r^{n-s-1}}\leq\liminf_{r\to0_+}\frac{V_{E, C}(r)-V_{E, C}(0)}{r^{n-s}}=\mathcal{SM}^s_{C}(E)_*.
    \end{equation*} The other inequality can be shown by analogy.
\end{proof}

Let us mention that Lemma \ref{lemma16} implies that the existence of $s$-dimensional $C$-anisotropic $\mathcal{S}$-content ensures the existence of $s$-dimensional $C$-anisotropic outer Minkowski content (and the values of both contents coincide).  
We will show that this is not always the case. Towards this end, we will compute the upper and lower $C$-anisotropic Minkowski content and $\mathcal{S}$-content for the \textit{Sierpinski gasket} similarly as in \cite{Rataj}, Example 3.3, and we will see that all the inequalities from the last part of Lemma \ref{lemma16} are in this case strict.

\begin{example}
    Let $D\coloneq \log_2(3)$ and $E\subseteq\mathbb{R}^2$ be the Sierpinski gasket, i.e. the self-similar set generated by the three similarities \begin{align*}
        \Phi_1(x)\coloneq \frac{1}{2} x,&& \Phi_2(x)\coloneq \frac{1}{2} x + \frac{1}{2}(1, 0)&& \text{and} &&\Phi_3(x)\coloneq \frac{1}{2} x + \left(\frac{1}{4}, \frac{\sqrt{3}}{4}\right).
    \end{align*} The set $E$ is compact, $\lambda^2$-negligible, and its Minkowski dimension is equal to $D$ (from the computation below it is clear that its $C$-anisotropic $\mathcal{S}$-dimension is equal to $D$ as well).
    
    Let us compute the $D$-dimensional lower and upper $C$-anisotropic Minkowski content and $\mathcal{S}$-content of $E$ for any $C\in\mathcal{C}^n_0$. We will see, similarly as in \cite{Rataj}, Example 3.3, that all of the inequalities from the last part of Lemma \ref{lemma16} can be strict. 

    Consider first an equilateral triangle $T$ in $\mathbb{R}^2$ with side length $s>0$. Further, let $v_1, v_2, v_3$ be the outward pointing normals to its sides. For the sake of simplicity, set $$u_1\coloneq h_C(-v_1)+h_C(-v_2)+h_C(-v_3),$$ $$u_2\coloneq h_C(v_1)+h_C(v_2)+h_C(v_3).$$ Then $$V_{T, C}(r)=\lambda^2(C)r^2+ s u_2 r+\frac{\sqrt{3}}{4}s^2,\; r\geq0$$and $$V_{\partial T, C}(r)=\left(\lambda^2(C)-\frac{u_1^2}{\sqrt{3}}\right)r^2+ \left(s u_1+s u_2\right) r\;\text{for all }r\in \left[0, \frac{\sqrt{3}s}{2 u_1}\right].$$ 

    We will now find the functions $V_{E, C}(\cdot)$ and $S_{E, C}(\cdot)$. Let $v_1, v_2, v_3$ be the outward pointing normals to the sides of the largest triangle the boundary of which is contained in $E$. Define $I_n \coloneq \left[ \frac{ 2^{-n-2}\sqrt{3}}{u_2}, \frac{ 2^{-n-1}\sqrt{3}}{u_2}\right)$ for $n\in\mathbb{N}$ and $I_0\coloneq \left[ \frac{\sqrt{3}}{4 u_2}, \infty\right)$. It can be easily deduced from the formul\ae{} for $V_{T, C}$ and $V_{\partial T, C}$ that $$V_{E, C}(r)=\lambda^2(C)r^2+ u_2 r+\frac{\sqrt{3}}{4},\;r\in I_0$$ and, more importantly, that for $n\in \mathbb{N}$ and $r\in I_n$ $$V_{E, C}(r)=\left(\lambda^2(C)-\frac{(3^n-1)u_2^2}{2\sqrt{3}}\right)r^2+ \left(\frac{3}{2}\right)^n u_2 r+\frac{\sqrt{3}}{4}\left(\frac{3}{4}\right)^n.$$ From these formul\ae{} we can deduce that $$S_{E, C}(r)=2\lambda^2(C)r+  u_2 \text{ for each } r\in I_0$$ and $$S_{E, C}(r)=\left(2\lambda^2(C)-\frac{(3^n-1)u_2^2}{\sqrt{3}}\right)r+ \left(\frac{3}{2}\right)^n u_2 \text{ for each } r\in I_n.$$

    Let $t_n(\alpha)\coloneq \frac{\alpha \sqrt{3}}{4 u_2}2^{-n}$, where $\alpha\in [1, 2)$, be a parametrization of the interval $I_n$. It follows that $t_n(\alpha)^{1-D}=\left(\frac{\alpha \sqrt{3}}{4 u_2}\right)^{1-D}\left(\frac{3}{2}\right)^n$. Now \begin{equation*}
    \begin{split}
         \frac{S_{E, C}\big(t_n(\alpha)\big)}{t_n(\alpha)^{1-D}}&= \left(\frac{\sqrt{3}}{4 u_2}\right)^D\left(\frac{2\sqrt{3} \lambda^2(C)+u_2^2}{3^n\sqrt{3}}-\frac{u_2^2}{\sqrt{3}}\right)\alpha^D + u_2\left(\frac{\sqrt{3}}{4 u_2}\right)^{D-1}\alpha^{D-1}\\&=c_n \alpha^D+b\alpha^{D-1},
    \end{split}
    \end{equation*}
    where we set \begin{align*}
        c_n\coloneq \left(\frac{\sqrt{3}}{4 u_2}\right)^D\left(\frac{2\sqrt{3} \lambda^2(C)+u_2^2}{3^n\sqrt{3}}-\frac{u_2^2}{\sqrt{3}}\right)&&\text{and}&&b\coloneq u_2\left(\frac{\sqrt{3}}{4 u_2}\right)^{D-1}.
    \end{align*} For the sake of brevity, define $f_n(\alpha)\coloneq c_n \alpha^D+b\alpha^{D-1}$ for $\alpha\in[1, 2]$. We can choose a sequence $\{\alpha_n\}_{n=1}^{\infty}\in [1, 2]^\mathbb{N}$ so that for each $n\in\mathbb{N}$ the continuous function $f_n$ attains its maximum at $\alpha_n$. As a consequence, we get \begin{equation}\label{plk}
        \lim_{n\to\infty}f_n(\alpha_n)=\lim_{n\to\infty}\frac{S_{E, C}\big(t_n(\alpha_n)\big)}{t_n(\alpha_n)^{1-D}}=(2-D)\mathcal{S}^D_C(E)^*.
    \end{equation}

Consider a function $f\colon\mathbb{R}^2\to\mathbb{R}$ given by the formula $f(x, y)\coloneq x^Dy+x^{D-1}b$. Further, set $c\coloneq \lim_{n\to\infty}c_n=-\sqrt{3}^{D-1}4^{-D}u_2^{2-D}$. We see that $\alpha_n\xrightarrow{n\to\infty}\alpha_{\textup{max}}$, where $\alpha_{\max}$ is the unique point in $[1, 2]$ where the maximum of the function $f(\cdot, c)$ is attained. It can be easily shown that $\alpha_{\textup{max}}=4\left(1-\frac{1}{D}\right)$. Since the function $f$ is continuous, we have $$\lim_{n\to\infty}f_n(\alpha_n)=\lim_{n\to\infty}f(\alpha_n, c_n)=f(\alpha_{\textup{max}}, c).$$ Using this and (\ref{plk}), we obtain \begin{equation*}
    \begin{split}
        \mathcal{S}^D_{C}(E)^*=\frac{f(\alpha_{\textup{max}}, c)}{(2-D)}=\frac{\alpha_{\textup{max}}^D c+\alpha_{\textup{max}}^{D-1}b}{(2-D)}=\frac{\left(\sqrt{3}-\frac{\sqrt{3}}{D}\right)^{D-1}u_2^{2-D}}{D(2-D)}\approx 1.170 u_2^{2-D}.
    \end{split}
\end{equation*}

        Let us find a sequence $\{\alpha^{\prime}_n\}_{n=1}^{\infty}\in [1, 2]^\mathbb{N}$ such that for each $n\in\mathbb{N}$ the function $f_n$ attains its minimum at $\alpha^{\prime}_n$. We can repeat the procedure from the preceding paragraph. It turns out that $\alpha^{\prime}_n\xrightarrow{k\to\infty}\alpha_{\textup{min}}=1$, where $\alpha_{\min}$ is the unique point in $[0, 1]$ where the function $f(\cdot, c)$ attains its minimum, and so \begin{equation*}
    \begin{split}
        \mathcal{S}^D_{C}(E)_*=\frac{f(\alpha_{\textup{min}}, c)}{(2-D)}=\frac{c+b}{(2-D)}=\frac{\sqrt{3}^{D-1}u_2^{2-D}}{4^D(2-D)}\approx 1.107 u_2^{2-D}.
    \end{split}
\end{equation*}
From this we immediately see that the $D$-dimensional $C$-anisotropic $\mathcal{S}$-content of $E$ does not exist.

Let us now compute both the lower and upper $D$-dimensional $C$-anisotropic Minkowski content of $E$. Notice first that $t_n(\alpha)^{2-D}=\left(\frac{\alpha \sqrt{3}}{4 u_2}\right)^{2-D}\left(\frac{3}{4}\right)^n$. We have the expression
\begin{equation*}
    \begin{split}
        \frac{V_{E, C}\big(t_n(\alpha)\big)}{t_n(\alpha)^{2-D}}&=\left(\frac{\sqrt{3}}{4 u_2}\right)^D\left(\frac{2\sqrt{3} \lambda^2(C)+u_2^2}{3^n2\sqrt{3}}-\frac{u_2^2}{2\sqrt{3}}\right)\alpha^D + u_2\left(\frac{\sqrt{3}}{4 u_2}\right)^{D-1}\alpha^{D-1}\\&+u_2\left(\frac{\sqrt{3}}{4 u_2}\right)^{D-1}\alpha^{D-2}\\&=\frac{1}{2}c_n \alpha^D+b\alpha^{D-1}+b\alpha^{D-2}. 
    \end{split}
\end{equation*}
Now, define $h_n(\alpha)\coloneq \frac{1}{2}c_n \alpha^D+b\alpha^{D-1}+b\alpha^{D-2}$ for $\alpha\in[1, 2]$ and for $n\in\mathbb{N}$ and $h(x, y)\coloneq  \frac{1}{2}x^D y+bx^{D-1}+bx^{D-2}$ for any $(x, y)\in\mathbb{R}^2$. Similarly as before, we can choose a sequence $\{\beta_n\}_{n=1}^{\infty}\in [1, 2]^\mathbb{N}$ so that for each $n\in\mathbb{N}$ the continuous function $h_n$ attains its maximum at $\beta_n$. Let $\beta_{\max}$ be the point in $[1, 2]$ where the continuous function $h(\cdot, c)$ attains its maximum. We obtain 
\begin{equation*}
        \mathcal{M}^D_{C}(E)^*=\lim_{n\to\infty}\frac{V_{E, C}\big(t_n(\beta_n)\big)}{t_n(\beta_n)^{2-D}}=\lim_{n\to\infty}h_n(\beta_n)=h(\beta_{\max}, c).
    \end{equation*}
It holds true that $\beta_{\max}=\frac{4}{D}\left(D-1+\sqrt{\frac{3}{2}D^2-3D+1}\right),$ whence \begin{equation*}
\begin{split}
    \mathcal{M}^D_{C}(E)^*=h(\beta_{\max}, c)\approx1.150 u_2^{2-D}.
\end{split}
\end{equation*}
In the case of $\mathcal{M}^D_{C}(E)_*$, we can argue the same way. The minimum of the function $h(\cdot, c)$ in $[1, 2]$ is attained at the point $\beta_{\min}=\frac{4}{D}\left(D-1-\sqrt{\frac{3}{2}D^2-3D+1}\right)$. So, \begin{equation*}
\begin{split}
    \mathcal{M}^D_{C}(E)_*=h(\beta_{\min}, c)\approx1.148 u_2^{2-D}.
\end{split}
\end{equation*}
From this, it follows that the $D$-dimensional $C$-anisotropic Minkowski content of $E$ does not exist.

\medskip
In summary, we see that all the inequalities from the last part of Lemma \ref{lemma16} are in this case strict, i.e. $$\mathcal{S}^D_{C}(E)_*<\mathcal{M}^D_{C}(E)_*<\mathcal{M}^D_{C}(E)^*<\mathcal{S}^D_{C}(E)^*.$$ 
\end{example}

We will prove similar inequalities as in \cite{Rataj} (Lemma 3.5 and Proposition 3.7).

\begin{lemma}\label{ooo}
    Let $f:[0, \infty)\to[0, \infty)$ be a Kneser function of order $n$. Then for any $s\in[0, n)$ we have $$\limsup_{r\to0_+}\frac{f(r)-f(0)}{r^{n-s}}\geq\frac{n-s}{n}\limsup_{r\to0_+}\frac{f^{+}(r)}{(n-s)r^{n-s-1}}.$$ In particular, if $E\in\mathcal{K}^n$ and $C\in\mathcal{C}^n_0$, then $$\mathcal{SM}^s_C(E)^*\geq \frac{n-s}{n}\mathcal{S}^s_C(E)^*.$$ If, in addition, $\lambda^n(E)=0$, we have $$\mathcal{M}^s_C(E)^*\geq \frac{n-s}{n}\mathcal{S}^s_C(E)^*.$$ Consequently, $\dim_\mathcal{M}(E)^*=\dim^C_\mathcal{S}(E)^*$ provided that $\lambda^n(E)=0$, and hence the value of  $\dim^C_\mathcal{S}(E)^*$ does not depend on the choice of $C$.
\end{lemma}
\begin{proof}
    We find a sequence $\{r_i\}_{i=1}^{\infty}\in \left(J_f\right)^{\mathbb{N}}$ with $r_i\searrow0$ as $i\to\infty$ (due to Lemma \ref{kneser}) so that \begin{equation*}
        a\coloneq \lim_{i\to\infty}\frac{f^{\prime}(r_i)}{(n-s)r_i^{n-s-1}}=\limsup_{r\to0_+}\frac{f^{+}(r)}{(n-s)r^{n-s-1}}.
    \end{equation*}
    Fix $i\in\mathbb{N}$. For any $r\in [r_{i+1}, r_i]\cap J_f$ we have \begin{equation*}
        \frac{f^{\prime}(r_{i})}{r_{i}^{n-1}}\leq\frac{f^{+}(r)}{r^{n-1}}\leq\frac{f^{\prime}(r_{i+1})}{r_{i+1}^{n-1}}
    \end{equation*} (see \cite{Stachó}, Theorem 1). Now (using Lemma \ref{kneser})\begin{equation}\label{dr}\begin{split}
            f(r_i)-f(0)&=\int\limits_{0}^{r_i}f^{\prime}\dif^1=\sum_{j=i}^{\infty}\;\int\limits_{r_{j+1}}^{r_j}f^{\prime}\dif^1\\&\geq\sum_{j=i}^{\infty}\;\int\limits_{r_{j+1}}^{r_j}f^{\prime}(r_j)\frac{r^{n-1}}{r_j^{n-1}}\dif^1(r)=\sum_{j=i}^{\infty}f^{\prime}(r_j)\frac{r_{j}^n - r_{j+1}^n}{n r_j^{n-1}}\\&=\sum_{j=i}^{\infty}\frac{f^{\prime}(r_j)}{(n-s)r_j^{n-s-1}}\frac{n-s}{n}\frac{r_{j}^n - r_{j+1}^n}{r_j^{s}}\\&\geq\sum_{j=i}^{\infty}\frac{f^{\prime}(r_j)}{(n-s)r_j^{n-s-1}}\frac{n-s}{n}\left(r_{j}^{n-s} - r_{j+1}^{n-s}\right).
    \end{split}
\end{equation}
If $b<a$, then there exists $j_0\in\mathbb{N}$ such that for any $j\geq j_0$ we have $$\frac{f^{\prime}(r_j)}{(n-s)r_j^{n-s-1}}\geq b.$$ Thus, for $i\geq j_0$ we obtain \begin{equation*}\begin{split}
    \frac{f(r_i)-f(0)}{r_i^{n-s}}&\myeqqq\sum_{j=i}^{\infty}\left(\frac{f^{\prime}(r_j)}{(n-s)r_j^{n-s-1}}\frac{n-s}{n}\frac{r_{j}^{n-s} - r_{j+1}^{n-s}}{r_i^{n-s}}\right)\\&\geq b\frac{n-s}{n} \sum_{j=i}^{\infty}\frac{r_{j}^{n-s} - r_{j+1}^{n-s}}{r_i^{n-s}}= b\frac{n-s}{n}\left(1-\frac{\lim_{j\to\infty}r_{j}^{n-s}}{r_i^{n-s}}\right)= b\frac{n-s}{n}.
\end{split}
\end{equation*}
Hence $$\limsup_{r\to0_+} \frac{f(r)-f(0)}{r^{n-s}}\geq b\frac{n-s}{n}.$$ Since $b<a$ was arbitrary, we get the desired inequality. 

The second part of the statement is now clear: set simply $f\coloneq V_{E, C}$.
\end{proof} 
We see that anisotropic upper Minkowski content and anisotropic upper $\mathcal{S}$-content behave asymptotically the same. However, in the case of anisotropic lower Minkowski content and anisotropic lower $\mathcal{S}$-content, the situation is very different. We have only the following inequality:
\begin{theorem}\label{omin}
    Let $E\in\mathcal{K}^n$ and $C\in\mathcal{C}^n_0$. Then for any $s\in[0, n]$ we have $$c(C, n, s)\mathcal{S}_C^{s\frac{n-1}{n}}(E)_*\geq \big(\mathcal{M}_C^{s}(E)_*\big)^{\frac{n-1}{n}},$$ where $c(C, n, s)$ is a positive constant depending only on $C$, $n$ and $s$ (and not on $E$). Consequently, $\dim^C_\mathcal{S}(E)_*\geq \frac{n-1}{n}\dim_\mathcal{M}(E)_*$ provided that $\lambda^n(E)=0$.
\end{theorem}
\begin{proof}
    Let $r>0$. By the anisotropic isoperimetric inequality (Theorem \ref{asas}) we have \begin{equation}\label{labels}
        n \lambda^n(C)^{\frac{1}{n}}V_{E, C}(r)^{\frac{n-1}{n}}\leq \textup{Per}_{h_{C}}(E_{r, C}; \mathbb{R}^n)=S_{E, C}(r).
    \end{equation} 
    Fix $s\in[0, n]$. For brevity, define $t\coloneq \frac{n-1}{n}s$. If $\mathcal{S}^t_C(E)_*=\infty$, there is nothing to prove. So, let us assume $\mathcal{S}^t_C(E)_*<\infty$. The inequality in (\ref{labels}) gives us \begin{equation*}
          \left(\frac{V_{E, C}(r)}{r^{n-s}}\right)^{\frac{n-1}{n}}\leq \frac{n-t}{n\lambda^n(C)^{\frac{1}{n}}}\frac{S_{E, C}(r)}{(n-t)r^{n-t-1}}.
    \end{equation*} Set $$c(C, n, s)\coloneq \frac{n-t}{n\lambda^n(C)^{\frac{1}{n}}}.$$ We find a sequence $\{r_i\}_{i=1}^{\infty}$ of positive real numbers converging to zero such that $$\lim_{i\to\infty}\frac{S_{E, C}(r_i)}{(n-t)r_i^{n-t-1}}=\mathcal{S}^t_C(E)_*.$$ For any real number $a>\mathcal{S}^t_C(E)_*$ there exists $i_0\in\mathbb{N}$ so that for each $i\geq i_0$ we have $$\frac{S_{E, C}(r_i)}{(n-t)r_i^{n-t-1}}\leq a.$$ Hence also $$ \left(\frac{V_{E, C}(r_i)}{r_i^{n-s}}\right)^{\frac{n-1}{n}}\leq c(C, n, s) a.$$ Thus \begin{equation*}
        \big(\mathcal{M}^s_C(E)_*\big)^{\frac{n-1}{n}}\leq\liminf_{i\to\infty}\left(\frac{V_{E, C}(r_i)}{r_i^{n-s}}\right)^{\frac{n-1}{n}}\leq c(C, n, s) a
    \end{equation*}
 Since $a>\mathcal{S}^t_C(E)_*$ was arbitrary, we obtain exactly $$\big(\mathcal{M}^s_C(E)_*\big)^{\frac{n-1}{n}}\leq c(C, n, s)\mathcal{S}^t_C(E)_*=c(C, n, s)\mathcal{S}_C^{s\frac{n-1}{n}}(E)_*,$$ which completes the proof.
    \end{proof}
\textbf{Summary:} For any compact and $\lambda^n$-negligible set $E\subseteq\mathbb{R}^n$, any convex body of full dimension $C\subseteq\mathbb{R}^n$ with $0\in\textup{int}(C)$, and any $s\in[0, n]$, we have the following inequalities:\begin{gather*}
\mathcal{S}^s_C(E)^*\geq\mathcal{M}^s_C(E)^*\geq\mathcal{M}^s_C(E)_*\geq\mathcal{S}^s_C(E)_*,\\\mathcal{M}^s_C(E)^*\geq\frac{n-s}{n}\mathcal{S}^s_C(E)^*\quad\text{and}\quad c(C, n, s) \mathcal{S}^{s\frac{n-1}{n}}_C(E)_*\geq\big(\mathcal{M}^s_C(E)_*\big)^{\frac{n-1}{n}}.
\end{gather*}We see that the equality $\dim_{\mathcal{M}}(E)^*=\dim^C_{\mathcal{S}}(E)^*$ always holds. However, equality $\dim_{\mathcal{M}}(E)_*=\dim^C_{\mathcal{S}}(E)_*$ does not hold in general. In \cite{Winter} (Theorem 2.2), it is shown that for any $0<s<m<1$ there exists a compact $\lambda^n$-negligible set $E\subseteq\mathbb{R}^n$ that satisfies $\dim^{B(0, 1)}_{\mathcal{S}}(E)_*=s+n-1$ and $\dim_{\mathcal{M}}(E)_*=m+n-1$.

\section{Further {P}roperties of $\mathcal{SM}^s_C$, $\mathcal{M}^s_C$ and $\mathcal{S}^s_C$}
\label{secFurther}
We will show that $\mathcal{SM}_C^s(E)$ and $\mathcal{S}_C^s(E)$, where $E\in\mathcal{K}^n$ with $\lambda^n(E)>0$, are of interest only for $s\in[n-1, n]$.
\begin{proposition}\label{okres}
   Let $E\in\mathcal{K}^n$ with $\lambda^n(E)>0$. Then for any $C\in\mathcal{C}^n_0$\begin{align*}
       \mathcal{S}^{n-1}_C(E)_*\geq n\lambda^n(C)^{\frac{1}{n}}\lambda^n(E)^{\frac{n-1}{n}}&&\text{and}&&\mathcal{SM}^s_C(E)=\mathcal{S}_C^s(E)=\infty,\text{ $s\in[0, n-1)$.}
   \end{align*}
   \begin{proof}
       The first inequality is a straightforward consequence of Theorem \ref{omin}. 
       
       Let us prove the second part. If $n=1$, there is nothing to prove. So, let $n>1$ and fix $s\in[0, n-1)$.  For brevity's sake, put $t\coloneq\frac{sn}{n-1}\in[0, n)$. Observe that from Theorem \ref{omin} we have $$\mathcal{S}^s_C(E)_*\geq c\left(C, n, t\right)^{-1}\big(\mathcal{M}_C^{t}(E)_*\big)^{\frac{n-1}{n}}.$$ However, the assumption $\lambda^n(E)>0$ implies $\mathcal{M}_C^{t}(E)_*=\infty$, whence, taking into account Lemma \ref{lemma16}, we get $\mathcal{S}^s_C(E)=\mathcal{SM}^s_C(E)=\infty$.
   \end{proof}
\end{proposition}
We will show what the limit $\lim_{r\to0_+}\kappa_{E, C}(r)$ is and that $\mathcal{S}^0_C(E)$ always exists.
\begin{proposition}
    Let $E\in\mathcal{K}^n$ and $C\in\mathcal{C}^n_0$. Then $$\frac{1}{n}\lim_{r\to0_+}\kappa_{E, C}(r)=\mathcal{S}^0_C(E)=\mathcal{SM}^0_C(E).$$ In particular, if $n>1$, then \begin{align}\label{formula}
        \frac{1}{n}\lim_{r\to0_+}\kappa_{E, C}(r)=\mathcal{H}^0(E)\lambda^n(C).
    \end{align}
    \begin{proof}
        If $E=\emptyset$, there is nothing to prove. So, let us assume that $E\neq\emptyset$. Observe $$\kappa_{E, C}(r)=\frac{S_{E, C}(r)}{r^{n-1}}\quad\text{for any $r>0.$}$$ We know that the function $\kappa_{E, C}(\cdot)$ decreases on the interval $(0, \infty)$ (see Lemma \ref{mnm}). Hence, we get the existence of the limit $$\mathcal{S}^0_{C}(E)=\lim_{r\to0_+}\frac{S_{E, C}(r)}{n r^{n-1}}\in\left[\lambda^n(C), \infty\right].$$ Lemma \ref{lemma16} implies that $\mathcal{S}^0_{C}(E)=\mathcal{SM}^0_{C}(E).$ 
        
        If $n>1$ and $\lambda^n(E)>0$, we have from Proposition \ref{okres} $\mathcal{S}^0_{C}(E)=\mathcal{SM}^0_{C}(E)=\infty$, which is, in particular, equal to $\mathcal{H}^0(E)\lambda^n(C)$. On the other hand, Lemma \ref{inter} gives us $\mathcal{S}^0_{C}(E)=\mathcal{M}^0_{C}(E)=\mathcal{H}^0(E)\lambda^n(C)$ for compacta with zero measure.\end{proof}
\end{proposition}
\begin{rem}
    It is clear that the formula (\ref{formula}) does not hold for $n=1$.
\end{rem}
Let us now prove a generalization of Theorem 3.4 from \cite{Villa2}. Since we work with ``$s\in[0, n]$'', we have to avoid countable $\mathcal{H}^{s}$-rectifiability. It turns out that (under some additional assumptions) the existence of $C$-anisotropic $s$-dimensional Minkowski content of each element of a family of countably many pairwise disjoint compact subsets covering the set up to an $\mathcal{H}^s$-negligible set ensures the existence of the content of the whole set. 
\begin{theorem}
    Let $E\in\mathcal{K}^n$, $C\in\mathcal{C}^n_0$ and $s\in[0, n]$. Suppose that\begin{itemize}
        \item There exists a countable pairwise disjoint family $\{E_k\}_{k\in\mathbb{N}}$ of compact subsets of $E$ such that $\mathcal{M}_C^s(E_k)$ exists for any $k\in\mathbb{N}$ and $\mathcal{H}^s\big( E\setminus \bigcup_{k\in\mathbb{N}}E_k\big)=0$, 
        \item There is a positive Radon measure $\eta$ on $\mathbb{R}^n$ which is absolutely continuous with respect to $\mathcal{H}^s$, so that for some $\gamma>0$ and for each $x\in E$ and any $t\in(0, 1)$ we have $$\eta\big(B(x, t)\big)\geq\gamma t^{s}. $$ Then $\mathcal{M}_C^s(E)$ exists and $$\mathcal{M}_C^s(E)=\sum_{k\in\mathbb{N}}\mathcal{M}_C^s(E_k).$$
    \end{itemize} 
    \begin{proof} We can find positive real numbers $a$ and $b$ such that $B(0, a)\subseteq C\subseteq B(0, b)$. Notice that the cases $s=n$ or $E=\emptyset$ are trivial. Therefore, we can suppose that $s<n$ and $E\neq\emptyset$.
    It suffices to show that \begin{align*}
        \mathcal{M}_C^s(E)_*\geq \sum_{k\in\mathbb{N}}\mathcal{M}_C^s(E_k)&&\text{and}&&\mathcal{M}_C^s(E)^*\leq\sum_{k\in\mathbb{N}}\mathcal{M}_C^s(E_k).
    \end{align*}
    
    Let us start with the first inequality. For any $K\in\mathbb{N}$ and $r>0$ we have the inclusion $$\bigcup_{k=1}^K (E_k\oplus rC)\subseteq E\oplus rC.$$ Hence \begin{equation}\label{nicos}\begin{split}\mathcal{M}_C^s(E)_*&=\liminf_{r\to0_+}\frac{\lambda^n(E\oplus rC)}{r^{n-s}}\geq\liminf_{r\to0_+}\frac{\lambda^n\left(\bigcup_{k=1}^K (E_k\oplus rC)\right)}{r^{n-s}}\\&\geq\sum_{k=1}^K \liminf_{r\to0_+}\frac{\lambda^n(E_k\oplus rC)}{r^{n-s}}=\sum_{k=1}^K \mathcal{M}_C^s(E_k).
        \end{split}
    \end{equation}
    Thus $$\mathcal{M}_C^s(E)_*\geq \sum_{k\in\mathbb{N}}\mathcal{M}_C^s(E_k).$$
 We see that if $E=\bigcup_{k=1}^K E_k$ for some $K\in\mathbb{N}$, then the limit inferior in (\ref{nicos}) is, in fact, the limit. So, we will suppose that $E\setminus \bigcup_{k=1}^K E_k\neq\emptyset$ for each $K\in\mathbb{N}$.

    Let us now prove the second inequality. Fix $\varepsilon\in(0, 1)$. There exists some $K\in\mathbb{N}$ such that $$\eta\left(E\setminus \bigcup_{k=1}^K E_k\right)<\varepsilon.$$ This follows from the fact that $\eta$ is absolutely continuous with respect to $\mathcal{H}^s$ (and from the continuity of $\eta$). Define for $r\in(0, r_{\varepsilon})$ $$E_{\varepsilon, r}\coloneq \left\{x\in E; \textup{dist}\left(x, \bigcup_{k=1}^K E_k\right)\geq 2ar\varepsilon^{1/n}\right\},$$ where $r_{\varepsilon}<a^{-1}\varepsilon^{-\frac{1}{n}}$ and $\eta\Big((E\oplus r_{\varepsilon}\varepsilon^{1/n}C)\setminus \bigcup_{k=1}^K E_k\Big)<\varepsilon$ as follows, again, from the continuity of $\eta$.
    
    We can cover the compact set $E_{\varepsilon, r}$ by a family of balls $\left\{B(i, \varepsilon, r)\right\}_{i\in I_{\varepsilon, r}}$, where $B(i, \varepsilon, r)\coloneq B\left(x_i, a r\varepsilon^{1/n}\right)$ and $x_i\in E_{\varepsilon, r}$, for all $i\in I_{\varepsilon, r}$, such that for some $N\in\mathbb{N}$ \begin{equation}\label{bez}
        \begin{split}
            \sum_{i\in I_{\varepsilon, r}} \eta\big(B(i, \varepsilon, r)\big)\leq N \eta\left(\left(E\oplus r\varepsilon^{1/n} C\right)\setminus \bigcup_{k=1}^K E_k \right).
        \end{split}
    \end{equation}
    The existence of such a family of balls can be justified as follows: We apply Theorem \ref{bes} on the compact set $E_{\varepsilon, r}$ and the constant function $x\mapsto ar\varepsilon^{1/n}$ on $E_{\varepsilon, r}$ and we get some $N\in\mathbb{N}$ and a family $\left\{B(i, \varepsilon, r)\right\}_{i\in I_{\varepsilon, r}}$ consisting of balls centered at some point of $E_{\varepsilon, r}$ such that \begin{align}\label{ded}
        E_{\varepsilon, r}\subseteq\bigcup_{i\in I_{\varepsilon, r}}B(i, \varepsilon, r)&&\text{and}&& \sum_{i\in I_{\varepsilon, r}}\chi_{B(i, \varepsilon, r)}\leq N.
    \end{align} Observe that the definition of the set $E_{\varepsilon, r}$ and the inclusion $B(0, a)\subseteq C$ ensure $$\bigcup_{i\in I_{\varepsilon, r}}B(i, \varepsilon, r)\subseteq \left(E_{\varepsilon, r}\oplus B\left(0, a r\varepsilon^{1/n}\right)\right)\setminus\bigcup_{k=1}^K E_k\subseteq \left(E\oplus r\varepsilon^{1/n}C\right)\setminus\bigcup_{k=1}^K E_k.  $$
    By integrating the second item of (\ref{ded}) with respect to the measure $\eta$ over the set $\left(E\oplus r\varepsilon^{1/n} C\right)\setminus \bigcup_{k=1}^K E_k$ we obtain exactly (\ref{bez}).
    
       By the assumption on $\eta$ we get \begin{equation*}\begin{split}
            \gamma a^s r^s\varepsilon^{\frac{s}{n}}\mathcal{H}^0(I_{\varepsilon, r})&=\sum_{i\in I_{\varepsilon, r}} \gamma\left( a r\varepsilon^{1/n}\right)^s\leq\sum_{i\in I_{\varepsilon, r}} \eta\Big(B\left(x_i, a r\varepsilon^{1/n}\right)\Big)\\&\weq N \eta\Big(\left(E\oplus r\varepsilon^{1/n} C\right)\setminus \bigcup_{k=1}^K E_k \Big)\leq N\varepsilon.
        \end{split}
    \end{equation*}
Rearrangement of the terms gives us the following estimate \begin{equation}\label{haus}
    \mathcal{H}^0(I_{\varepsilon, r})\leq\frac{N\varepsilon^{\frac{n-s}{n}}}{\gamma a^s r^s}.
\end{equation}
Observe that since $C\subseteq B(0, b)$,\begin{equation*}\begin{split}
        E_{\varepsilon, r}\oplus r C&\subseteq\bigcup_{i\in I_{\varepsilon, r}} B\big(x_i, a r\varepsilon^{1/n}\big)\oplus B(0, br)\\&\subseteq\bigcup_{i\in I_{\varepsilon, r}} B\big(x_i, b r(1+\varepsilon^{1/n})\big).
\end{split}
\end{equation*}
Therefore, \begin{equation}\label{kik}
    \begin{split}
        \lambda^n\big(E_{\varepsilon, r}\oplus rC\big)&\leq\sum_{i\in I_{\varepsilon, r}} \lambda^n\Big(B\big(x_i, b r(1+\varepsilon^{1/n})\big)\Big)\\&=\mathcal{H}^0(I_{\varepsilon, r})\omega_n b^n r^n (1+\varepsilon^{1/n})^n\\&\myequh\omega_n b^n r^n (1+\varepsilon^{1/n})^n\frac{N\varepsilon^{\frac{n-s}{n}}}{\gamma a^s r^s}\\&\leq r^{n-s}\frac{N\omega_n b^n 2^n\varepsilon^{\frac{n-s}{n}}}{\gamma a^s}.
    \end{split}
\end{equation}
We have the following chain of inclusions:
\begin{equation}\label{incl}
    \begin{split}
        E\oplus rC&\subseteq \left(E\setminus E_{\varepsilon, r}\right)\oplus rC\cup \left( E_{\varepsilon, r}\oplus rC\right)\\&\subseteq \bigcup_{k=1}^K\left(E_k\oplus 2ar\varepsilon^{1/n}C\oplus rC\right)\cup(E_{\varepsilon, r}\oplus rC),
    \end{split}
\end{equation}
where we used that $$E\setminus E_{\varepsilon, r}\subseteq\left\{x\in E : \dist\Big(x, \bigcup_{k=1}^K E_k\Big)\leq 2ar\varepsilon^{1/n}\right\}.$$ 
Finally, \begin{equation*}
    \begin{split}
        \mathcal{M}_C^s(E)^*&=\limsup_{r\to0_+}\frac{\lambda^n(E\oplus rC)}{r^{n-s}}\\&\myequuh\limsup_{r\to0_+}\frac{\lambda^n(E_{\varepsilon, r}\oplus rC)}{r^{n-s}}+\sum_{k=1}^K\limsup_{r\to0_+}\frac{\lambda^n\big(E_k\oplus r(1+2a\varepsilon^{1/n}) C\big)}{r^{n-s}}\\&\myequuuh \frac{N\omega_n b^n 2^n\varepsilon^{\frac{n-s}{n}}}{\gamma a^s}+(1+2a\varepsilon^{1/n})^{n-s}\sum_{k=1}^K\mathcal{M}_C^s(E_k)\\&\leq \frac{N\omega_n b^n 2^n\varepsilon^{\frac{n-s}{n}}}{\gamma a^s}+(1+2a\varepsilon^{1/n})^{n-s}\sum_{k\in\mathbb{N}}\mathcal{M}_C^s(E_k).
    \end{split}
\end{equation*}
By sending $\varepsilon\searrow0$, we get the desired inequality $$\mathcal{M}_C^s(E)^*\leq\sum_{k\in\mathbb{N}}\mathcal{M}_C^s(E_k),$$ which completes the proof.\end{proof}
\end{theorem}


\begin{thebibliography}{KKP2}
\bibitem[ACV]{Villa4}Luigi Ambrosio, Andrea Colesanti and Elena Villa. Outer {M}inkowski content for some classes of closed sets. \textit{Mathematische Annalen}, 342:727--748, 2008.
\bibitem[AFP]{ambrosio}Luigi Ambrosio, Nicola Fusco and Diego Pallara. Functions of bounded variation and free discontinuity problems. \textit{Clarendon Press}, 2000.
\bibitem[CLL]{chambolle} Antonin Chambolle, Stefano Lisini and Luca Lussardi. A remark on the anisotropic outer {M}inkowski content. \textit{Advances in Calculus of Variations}, 7:241--266, 2014.
\bibitem[CLV]{Villa}Antonin Chambolle, Luca Lussardi and Elena Villa. Anisotropic tubular neighborhoods of sets. \textit{Mathematische Zeitschrift}, 299:1--18, 2021.
\bibitem[EG]{evans}Lawrence C. Evans and Ronald F. Gariepy. \textit{CRC Press, Taylor and Francis Group}, 2015.
\bibitem[Fo]{Fonseca} Irene Fonseca. The {W}ulff {T}heorem {R}evisited. \textit{Proceedings: Mathematical and Physical Sciences}, 432:125--145, 1991.
\bibitem[He]{Heil}Christopher Heil. Absolute {C}ontinuity and the {B}anach--{Z}aretsky {T}heorem.  \textit{Excursions in Harmonic Analysis}, 6:27--51, 2021.
\bibitem[Kn]{Kneser}Martin Kneser. {Ü}ber den {R}and von {P}arallelkörpern. \textit{Mathematische Nachrichten}, 5:241--251, 1951.
\bibitem[KR]{Rataj5}Markus Kiderlen and Jan Rataj. On the (outer) Minkowski content with lower-dimensional structuring element. arXiv preprint, https://arxiv.org/abs/2504.03339, 2025.
\bibitem[LV]{Villa2}Luca Lussardi and Elena Villa. A general formula for the anisotropic outer {M}inkowski content of a set. \textit{Proc. Roy. Soc. Edinburgh Sect. A}, 146:393--413, 2016.
\bibitem[Ma]{ASE}Francesco Maggi. Sets of finite perimeter and geometric variational problems : an introduction to geometric measure theory.
\textit{Cambridge University Press}, 2012.
\bibitem[RW]{Rataj}Jan Rataj and Steffen Winter. On volume and surface area of parallel sets. \textit{Indiana University Mathematics Journal}, 59:1661–-1686, 2010.
\bibitem[Sch]{rolf} Rolf Schneider. Convex bodies : the Brunn-{M}inkowski theory. \textit{Cambridge University Press}, 2014.

\bibitem[Si]{simon}Leon Simon. Lectures on Geometric Measure Theory. \textit{Centre for Mathematical Analysis, Australian National University}, 1984.
\bibitem[St]{Stachó}László L. Stachó. On the volume function of parallel sets. \textit{Acta Sci. Math.}, 38:365–-374, 1976.
\bibitem[Vi]{vrv}Elena Villa. On the outer {M}inkowski content of sets. \textit{Annali di Matematica Pura ed Applicata}, 188:619--630, 2009.
\bibitem[Wi]{Winter}Steffen Winter. Lower {S}-dimension of fractal sets. \textit{Journal of Mathematical Analysis and Applications}, 375:467--477, 2011.
\end{thebibliography}
\end{document}